\documentclass{amsart}
\usepackage[all]{xy}
\usepackage{enumerate}
\usepackage{mathrsfs}
\usepackage{amssymb}
\usepackage{graphicx}

\numberwithin{equation}{subsection}

\newcommand{\nc}{\newcommand}
\nc\rnc{\renewcommand}

\theoremstyle{plain}
\newtheorem{theorem}{Theorem}[section]
\newtheorem{lemma}[theorem]{Lemma}
\newtheorem{corollary}[theorem]{Corollary}
\newtheorem{proposition}[theorem]{Proposition}
\newtheorem{conjecture}[theorem]{Conjecture}
\theoremstyle{definition}

\theoremstyle{remark}
\newtheorem{remark}[theorem]{Remark}

\newcommand\Aut{\operatorname{Aut}}

\newcommand\Hom{\operatorname{Hom}}

\newcommand\GL{\operatorname{GL}}
\newcommand\SL{\operatorname{SL}}

\newcommand\id{\operatorname{id}}

\newcommand\im{\operatorname{im}}

\newcommand\coker{\operatorname{coker}}

\newcommand\Z{\mathbb{Z}}

\newcommand\C{\mathbb{C}}

\newcommand\R{\mathbb{R}}
\newcommand\Q{\mathbb{Q}}

\newcommand\h{\mathfrak{h}}
\newcommand\gpS{\mathfrak{S}}

\newcommand\red[1]{{\color{red}#1}}
\renewcommand\red{}

\newcommand{\GLnZ}{\GL(n,\Z)}

\nc\FAb{\mathbf{FAb}}

\nc\xto[1]{{\overset{#1}{\longrightarrow}}}
\nc\yto[1]{{\underset{#1}{\longrightarrow}}}
\nc\xyto[2]{{\overset{#1}{\underset{#2}{\longrightarrow}}}}
\nc\ad{{\operatorname{ad}}}
\nc\even{{\operatorname{even}}}
\nc\ev{{\operatorname{ev}}}
\nc\coev{{\operatorname{coev}}}
\nc\odd{{\operatorname{odd}}}
\nc\half{{\frac12}}
\nc\halfof[1]{{\frac{#1}2}}
\nc\onto\twoheadrightarrow
\nc\projto{\underset{\text{proj}}{\longrightarrow}}
\nc\no[1]{}
\nc\ok{\comm{ok?}}
\nc\ho{{\hat\otimes }}
\nc\plim{\varprojlim}
\nc\np{\newpage}

\nc\bfA{\mathbf{A}} \nc\bbA{\mathbb{A}} \nc\calA{\mathcal{A}}
\nc\bfB{\mathbf{B}} \nc\bbB{\mathbb{B}} \nc\calB{\mathcal{B}}
\nc\bfC{\mathbf{C}} \nc\bbC{\mathbb{C}} \nc\calC{\mathcal{C}}
\nc\bfD{\mathbf{D}} \nc\bbD{\mathbb{D}} \nc\calD{\mathcal{D}}
\nc\bfE{\mathbf{E}} \nc\bbE{\mathbb{E}} \nc\calE{\mathcal{E}}
\nc\bfF{\mathbf{F}} \nc\bbF{\mathbb{F}} \nc\calF{\mathcal{F}}
\nc\bfG{\mathbf{G}} \nc\bbG{\mathbb{G}} \nc\calG{\mathcal{G}}
\nc\bfH{\mathbf{H}} \nc\bbH{\mathbb{H}} \nc\calH{\mathcal{H}}
\nc\bfI{\mathbf{I}} \nc\bbI{\mathbb{I}} \nc\calI{\mathcal{I}}
\nc\bfJ{\mathbf{J}} \nc\bbJ{\mathbb{J}} \nc\calJ{\mathcal{J}}
\nc\bfK{\mathbf{K}} \nc\bbK{\mathbb{K}} \nc\calK{\mathcal{K}}
\nc\bfL{\mathbf{L}} \nc\bbL{\mathbb{L}} \nc\calL{\mathcal{L}}
\nc\bfM{\mathbf{M}} \nc\bbM{\mathbb{M}} \nc\calM{\mathcal{M}}
\nc\bfN{\mathbf{N}} \nc\bbN{\mathbb{N}} \nc\calN{\mathcal{N}}
\nc\bfO{\mathbf{O}} \nc\bbO{\mathbb{O}} \nc\calO{\mathcal{O}}
\nc\bfP{\mathbf{P}} \nc\bbP{\mathbb{P}} \nc\calP{\mathcal{P}}
\nc\bfQ{\mathbf{Q}} \nc\bbQ{\mathbb{Q}} \nc\calQ{\mathcal{Q}}
\nc\bfR{\mathbf{R}} \nc\bbR{\mathbb{R}} \nc\calR{\mathcal{R}}
\nc\bfS{\mathbf{S}} \nc\bbS{\mathbb{S}} \nc\calS{\mathcal{S}}
\nc\bfT{\mathbf{T}} \nc\bbT{\mathbb{T}} \nc\calT{\mathcal{T}}
\nc\bfU{\mathbf{U}} \nc\bbU{\mathbb{U}} \nc\calU{\mathcal{U}}
\nc\bfV{\mathbf{V}} \nc\bbV{\mathbb{V}} \nc\calV{\mathcal{V}}
\nc\bfW{\mathbf{W}} \nc\bbW{\mathbb{W}} \nc\calW{\mathcal{W}}
\nc\bfX{\mathbf{X}} \nc\bbX{\mathbb{X}} \nc\calX{\mathcal{X}}
\nc\bfY{\mathbf{Y}} \nc\bbY{\mathbb{Y}} \nc\calY{\mathcal{Y}}
\nc\bfZ{\mathbf{Z}} \nc\bbZ{\mathbb{Z}} \nc\calZ{\mathcal{Z}}
\nc\bfone {{\mathbf 1}}

\nc\Vect{\mathbf{Vect}}
\nc\Sets{\mathbf{Sets}}
\nc\Mod{\mathbf{Mod}}
\nc\Cat{\mathbf{Cat}}

\nc\ul{\underline}
\nc\simeqto{\overset{\simeq}{\longrightarrow }}
\nc\ct{\overset{\cong}{\longrightarrow }}
\nc\mt{\mapsto}
\nc\hr{\medskip\hrule\medskip}
\nc\trl{\triangleleft}
\nc\trr{\triangleright}

\nc\xysquare[8]{\xymatrix{
    #1 \ar[r]#5 \ar[d]#6 & #2 \ar[d]#7 \\
    #3 \ar[r]#8          & #4
  }
  }

\nc\Ob{\operatorname{Ob}}
\nc\Mor{\operatorname{Mor}}

\nc\al{\alpha}
\nc\be{\beta}
\nc\la{\lambda}
\nc\ot{\otimes}
\nc\ott{\ot\cdots\ot}
\nc\Sp{\operatorname{Sp}}
\nc\SO{\operatorname{SO}}

\nc\HH{\mathrm{HH}}

\nc\ci{\circ}
\nc\sq{\square}

\nc\incl{\mathrm{incl}}

\nc\ol{\overline}
\nc\sqcups{\sqcup\cdots\sqcup}
\nc\congto{\overset{\cong}{\to}}

\nc\hide[1]{}
\nc\blue[1]{{\textcolor[rgb]{0,0,.9}{#1}}}
\nc\bluen[1]{\blue{[[#1]]}}
\nc\bnote{\bluen}
\nc\redn[1]{\red{[[#1]]}}

\nc\comment[1]{\marginpar{\tiny #1}}
\nc\len{\operatorname{len}}

\nc\VIC{\mathrm{VIC}}

\nc\GLnQ{\GL(n,\Q)}

\nc\nKMP{n_{\mathrm{KMP}}}
\nc\nB{n_{\mathrm{B}}}

\nc\Pol{\mathcal{P}ol}
\nc\VICmod{\VIC\text{-}\mathrm{mod}}
\nc\colim{\operatorname{colim}}

\nc\ulla{{\ul\lambda}}
\nc\Sz{\mathsf{Sz}}
\nc\St{\mathsf{St}}

\nc\SF{\mathrm{SF}}
\nc\AutFn{\Aut(F_n)}

\let\copybigwedge\bigwedge
\renewcommand\bigwedge{\copybigwedge\nolimits}

\title[On Borel's stable range of the twisted cohomology of $\GL(n,\Z)$]{On Borel's stable range\\ of the twisted cohomology of $\GL(n,\Z)$}

\author{Kazuo Habiro}
\author{Mai Katada}
\address{Department of Mathematics, Kyoto University, Kyoto 606-8502, Japan}
\email{habiro@math.kyoto-u.ac.jp}
\email{katada.mai.36s@st.kyoto-u.ac.jp}

\date{December 18, 2022}


\begin{document}
\maketitle

\begin{abstract}
Borel's stability and vanishing theorem gives the stable cohomology of $\GL(n,\Z)$ with coefficients in algebraic $\GL(n,\Z)$-representations. 
We compute the improved stable range that Borel remarked about. 
In order to further improve Borel's stable range, we adapt the method of Kupers--Miller--Patzt to any algebraic $\GL(n,\Z)$-representations.
\end{abstract}

\setcounter{tocdepth}{1}

\section{Introduction}

Borel proved the stability of the rational cohomology of $\GL(n,\Z)$ and computed the stable cohomology \cite{Borel1}.
He also proved the vanishing of the stable cohomology of $\GL(n,\Z)$ with coefficients in non-trivial algebraic $\GL(n,\Z)$-representations \cite{Borel2}.
He gave constants for the stable ranges and remarked about improved stable ranges, but he did not compute these stable ranges explicitly except for a few types of representations.

Li and Sun \cite{Li-Sun} improved Borel's stable ranges and obtained stable ranges that are independent of types of representations.
For coefficients in polynomial $\GLnZ$-representations, Kupers, Miller and Patzt \cite{Kupers-Miller-Patzt}
improved the stable ranges by using arguments on polynomial $\VIC$-modules.

In this paper, we compute the improved stable range that Borel remarked about. We also adapt Kupers, Miller and Patzt's argument to coefficients in algebraic $\GLnZ$-representations indexed by \emph{bipartitions}, i.e., pairs of partitions.
Our results are weaker than Li and Sun's. However, the methods are very different and we think that it is still worth publishing these results.

\subsection{Stable range for the cohomology of $\GL(n,\Z)$}
Borel studied the stable twisted cohomology of $\GL(n,\Z)$.
The Borel stability theorem \cite{Borel1} gives the stable rational cohomology of $\GLnZ$, and the Borel vanishing theorem \cite{Borel2} gives the vanishing of the cohomology of $\GLnZ$ with coefficients in non-trivial algebraic $\GL(n,\Z)$-representations in a stable range
(see Section \ref{alg-glnz-rep} for the definition of algebraic $\GL(n,\Z)$-representations).

Li and Sun \cite{Li-Sun} gave an improved stable range, which is independent of the types of $\GLnZ$-representations.

\begin{theorem}[Borel \cite{Borel1, Borel2}, Li--Sun \cite{Li-Sun}]
\label{Borel-Li-Sun}
(1) For each integer $n\ge 1$, the algebra map $$H^*(\GL(n+1,\Z),\Q)\to H^*(\GL(n,\Z),\Q)$$ induced by the inclusion $\GL(n,\Z)\hookrightarrow \GL(n+1,\Z)$ is an isomorphism in $*\le n-2$.
Moreover, we have an algebra isomorphism
\begin{gather*}
    \varprojlim_n H^*(\GL(n,\Z),\Q)\cong \bigwedge_{\Q} (x_1, x_2,\ldots), \quad \deg x_i=4i+1
\end{gather*}
in degree $*\le n-2$.

(2)
 Let $V$ be an algebraic $\GL(n,\Q)$-representation such that $V^{\GL(n,\Q)}=0$.
 Then we have 
 \begin{gather*}
     H^p(\GL(n,\Z),V) = 0 \quad \text{for $p\le n-2$}.
 \end{gather*}
\end{theorem}

Kupers, Miller and Patzt \cite{Kupers-Miller-Patzt} obtained the stable ranges for coefficients in irreducible polynomial $\GL(n,\Z)$-representations corresponding to partitions by using arguments of $\VIC$-modules. 
In particular, their stable range for the rational cohomology is wider by $1$
than that of Li and Sun. 

\begin{theorem}[Borel \cite{Borel1, Borel2}, Kupers--Miller--Patzt \cite{Kupers-Miller-Patzt}]
We have
\begin{gather*}
 H^*(\GL(n,\Z), \Q)\cong \bigwedge_{\Q} (x_1, x_2,\ldots), \quad \deg x_i=4i+1
\end{gather*}
in degree $*\le n-1$.
\end{theorem}


We explicitly compute the improved stable range that Borel remarked about, and furthermore, we improve the stable range by adapting the method of Kupers, Miller and Patzt \cite{Kupers-Miller-Patzt}, although these stable ranges depend on the types of representations and are weaker than the stable range given by Li and Sun.
For a bipartition $\ul\lambda=(\lambda,\lambda')$, let $V_{\ul\la}$ denote the (irreducible or zero) algebraic $\GL(n,\Z)$-representation corresponding to $\ul\la$ (see Section \ref{alg-glnz-rep}).

\begin{theorem}[Corollary \ref{corollaryhomologyofGL} and Theorem \ref{KMPtheorem}, weaker than Theorem \ref{Borel-Li-Sun}]
Let $\ul\lambda\neq (0,0)$ be a bipartition. Then we have
  $$H^p(\GL(n,\Z),V_{\ul\lambda})=0$$
  for $n\ge n_0(\ul\lambda,p)$.
\end{theorem}

Here the constant $n_0(\ul\lambda,p)$ is defined as follows.
For a partition $\lambda$, let $|\lambda|$ and $l(\lambda)$ denote the size and length of $\lambda$, respectively.
Let $\ul\lambda=(\la,\la')$ be a bipartition and $p$ a non-negative integer.
Let $|\ulla|=|\la|+|\la'|$ and $\deg \ulla=|\la|-|\la'|$, and set
$$n_0(\ul\lambda,p)=\min\{\nKMP(\ul\lambda,p), \nB(\ul\lambda,p)\},$$
where
$$
\nKMP(\ul\lambda,p)=
\begin{cases}
p+1+|\ul\lambda| & (\text{if } \la=0\text{ or }\la'=0)\\
p+1+2|\ul\lambda| & (\text{otherwise})
\end{cases}
$$
and 
$$\nB(\ul\lambda,p)=\max(2p+2, 2|\deg\ul\lambda|+1, 2l(\lambda), 2l(\lambda')).$$

\begin{remark}
Let us compare the value of $\nKMP(\ul\lambda,p)$ and $\nB(\ul\la,p)$.
For a fixed $\ul\la$, we have $\nKMP(\ul\la,p)<\nB(\ul\la,p)$ for all but finitely many $p$.
If $p$ is relatively small with respect to $\ul\la$ then we sometimes have $\nB(\ul\la,p)<\nKMP(\ul\la,p)$.
For example, we have $\nKMP((4,4),1)=18$ and $\nB((4,4),1)=4$.
Note that Theorem \ref{Borel-Li-Sun} gives a better bound $3$ in this case.
\end{remark}

\begin{remark}
\label{rem1}
This paper stemmed from the first version of \cite{Habiro-Katada} (with a different title), which included the two approaches to improve Borel's stable ranges described in this paper.
After the first version of \cite{Habiro-Katada} appeared on the arXiv, 
Oscar Randal-Williams informed us of the result of Li and Sun about the improvement of the Borel theorem \cite{Li-Sun}.
Since our results for Borel's stable range 
turned out to be weaker than Li and Sun's result, we have decided to remove these results from \cite{Habiro-Katada} and to rely on Li and Sun's result there.
\end{remark}

\subsection{Organization of the paper}

The rest of this paper is organized as follows.
In Section \ref{alg-glnz-rep}, we recall some facts about representation theory of $\GL(n,\Q)$.
In Section \ref{homologyofGL}, we recall Borel's stability and vanishing theorem for $\GL(n,\Z)$ and compute the improved stable range that Borel remarked about for irreducible algebraic representations.
In Section \ref{secImprovedVanishing}, we improve the stable range by using the arguments of Kupers, Miller and Patzt \cite{Kupers-Miller-Patzt}.

\subsection*{Acknowledgements} The authors thank Geoffrey Powell and Oscar Randal-Williams for helpful comments.
K.H. was supported in part by JSPS KAKENHI Grant Number 18H01119 and 22K03311.
M.K. was supported in part by JSPS KAKENHI Grant Number JP22J14812.

\section{Algebraic $\GL(n,\Z)$-representations}
\label{alg-glnz-rep}

Let $n\ge1$ be an integer.
A \emph{polynomial $\GL(n,\Q)$-representation} is a finite-dimensional $\Q[\GL(n,\Q)]$-module $V$ such that after choosing a basis for $V$, the $(\dim V)^2$ coordinate functions are polynomial in the $n^2$ variables.
A $\GL(n,\Q)$-representation is called \emph{algebraic} if the coordinate functions are rational functions.
See \cite{Fulton-Harris} for some facts from representation theory.

As is well known, irreducible polynomial $\GL(n,\Q)$-representations are classified by partitions with at most $n$ parts.
A \emph{partition} $\lambda=(\lambda_1,\lambda_2,\dots,\lambda_l)$ is a weakly decreasing sequence of non-negative integers.
The \emph{length} $l(\lambda)$ of $\lambda$ is defined by $l(\lambda)=\max(\{0\}\cup\{i\mid \lambda_{i}> 0\})$
and the \emph{size} $|\lambda|$ of $\lambda$ is defined by $|\lambda|=\lambda_1+\cdots+\lambda_{l(\lambda)}$.

We denote by $H=H(n)=\Q^n$ the standard representation of $\GL(n,\Q)$. 
In the following, we usually omit $(n)$.
For a partition $\lambda$, 
the Specht module $S^\lambda$ for $\lambda$ is an irreducible representation of $\gpS_{|\lambda|}$ defined by using the Young symmetrizer associated to $\lambda$.
Define a $\GL(n,\Q)$-representation $$V_\lambda=V_\lambda(n)= H^{\otimes |\lambda|}\otimes_{\Q[\gpS_{|\lambda|}]}S^\lambda.$$
If $l(\lambda)\le n$, then $V_\lambda$ is an irreducible polynomial $\GL(n,\Q)$-representation.
Otherwise, 
we have $V_{\lambda}=0$.

Let $p,q\ge0$ be integers.
We set $H^{p,q}=H^{\otimes p}\otimes (H^*)^{\otimes q}$.
For a pair $(i,j)\in \{1,\dots,p\}\times \{1,\dots,q\}$, we define the
\emph{contraction map} 
\begin{gather}
\label{cij}
c_{i,j}:H^{p,q}\rightarrow H^{p-1,q-1}
\end{gather}
by
\begin{gather*}
\begin{split}
&c_{i,j}((v_1\otimes\cdots\otimes v_{p}) \otimes(f_1\otimes\cdots\otimes f_{q}))\\
&\quad=  \langle v_i,f_j\rangle
    (v_1\otimes\cdots\widehat{v_{i}}\cdots\otimes v_{p})\otimes (f_1\otimes\cdots\widehat{f_{j}}\cdots\otimes f_{q})
\end{split}
\end{gather*}
 for $v_1,\ldots, v_p\in H$ and $f_1,\ldots, f_q\in H^*$,
where the dual pairing $\langle -,-\rangle:H\otimes H^{*}\rightarrow \Q$ is defined by $\langle v,f\rangle=f(v)$.

The \emph{traceless part} $H^{\langle p,q\rangle}$ of $H^{p,q}$ is defined by
$$
  H^{\langle p,q\rangle}=\bigcap_{(i,j)\in \{1,\dots,p\}\times \{1,\dots,q\}} \ker c_{i,j}\subset H^{p,q},
$$
which is a $\GL(n,\Q)$-subrepresentation of $H^{p,q}$.

A \emph{bipartition} is a pair $\ul{\lambda}=(\lambda,\lambda')$ of two partitions $\lambda$ and $\lambda'$.
The \emph{length} $l(\ul{\lambda})$ of the bipartition $\ul{\lambda}$ is defined by $l(\ul\lambda)=l(\lambda)+l(\lambda')$.
The \emph{degree} of $\ul\lambda$ is defined by
$\deg\ul\lambda=|\lambda|-|\lambda'|\in\Z$, and the \emph{size} of $\ul\lambda$ by $|\ul\lambda|=|\lambda|+|\lambda'|$.
We define the \emph{dual} of $\ul\lambda$ by $\ul\lambda^*=(\lambda',\lambda)$.

We associate to each bipartition $\ul\lambda=(\lambda,\lambda')$ a $\GL(n,\Z)$-representation
\begin{gather}
\label{defVlambda}
V_{\ul\lambda}=V_{\ul\lambda}(n)=H^{\langle p,q\rangle}\otimes_{\Q[\gpS_p\times \gpS_q]}(S^{\lambda}\otimes S^{\lambda'}),
\end{gather}
where $p=|\lambda|$ and $q=|\lambda'|$.
If $l(\ul\lambda)\le n$, then $V_{\ul\lambda}$ is an irreducible algebraic $\GLnQ$-representation. Otherwise, we have $V_{\ul\lambda}=0$.
It is well known that irreducible algebraic $\GL(n,\Q)$-representations are classified by bipartitions $\ul\la$ with $l(\ul\la)\le n$.

The traceless part $H^{\langle p,q\rangle}$ of $H^{p,q}$ admits the following direct-sum decomposition
as a $\Q[\GL(n,\Q)\times (\gpS_p\times \gpS_q)]$-module
\begin{gather}\label{kercontraction}
    H^{\langle p,q\rangle}=\bigoplus_{\substack{\ul\lambda=(\lambda,\lambda'): \text{bipartition with} \\ l(\ul\lambda)\le n,\;|\lambda|=p,\;|\lambda'|=q}} V_{\ul\lambda}\otimes (S^{\lambda}\otimes S^{\lambda'}).
\end{gather}
(See \cite[Theorem 1.1]{Koike}.)

Note that we have $\GL(n,\Q)$-isomorphisms
$\det \cong \bigwedge^n V \cong V_{(1^n)}$, where
$\det$ denote the determinant representation, and $(1^n)=(1,\dots,1)$ consists of $n$ copies of $1$.
For any bipartition $\ul\lambda=(\la,\la')$ with $l(\ul\la)\le n$, we have an isomorphism
$$V_{\ul\lambda}\cong V_{\mu}\otimes {\det}^k,$$
for some partition $\mu$ with at most $n$ parts and an integer $k$ such that
$$
(\lambda_1,\ldots,\lambda_{l(\lambda)},0,\ldots,0,-\lambda'_{l(\lambda')},\ldots,-\lambda'_1)=(\mu_1+k,\ldots,\mu_n+k).
$$

By an \emph{algebraic $\GL(n,\Z)$-representation}, we mean the restriction of an algebraic $\GL(n,\Q)$-representation to $\GLnZ$.
Note that ${\det}^2$ is trivial as a $\GL(n,\Z)$-representation.
It follows that any irreducible algebraic $\GL(n,\Z)$-representation is obtained from an irreducible polynomial $\GL(n,\Q)$-representation by restriction to $\GL(n,\Z)$.

\section{Borel's improved stable range}\label{homologyofGL}

In \cite{Borel1,Borel2}, Borel computed the cohomology $H^p(\Gamma,V)$ of an arithmetic group $\Gamma$ with coefficients in an algebraic $\Gamma$-representation $V$ in a stable range $$p\le N(\Gamma,V)=\min(M(\Gamma(\R),V),C(\Gamma(\Q),V)),$$ where $M(\Gamma(\R),V)$ and $C(\Gamma(\Q),V)$ are constants depending only on $\Gamma$ and $V$.
For $\Gamma=\SL(n,\Z)$, 
we have
$M(\SL(n,\R),V)\ge n-2$.
Borel did not compute the constant $C(\SL(n,\Q),V)$ explicitly except for a few types of representations.
Recently, Krannich and Randal-Williams \cite{Krannich-RW} gave an estimation of $C(\SL(n,\Q),V)$.

Borel remarked that one can replace the constant $C(\Gamma(\Q),V)$ by an improved constant $C'(\Gamma(\Q),V)\ge C(\Gamma(\Q),V)$ \cite[Remark 3.8]{Borel2}.
In this section, we give an estimation of Borel's improved constant for $\Gamma=\SL(n,\Z)$.
The constant $C'(\SL(n,\Q),V)$ depends not only on $n$ but also on the type of the representation $V$,
unlike the cases of $\Sp(2n,\Z)$ and $\SO(n,n;\Z)$ which were determined by Tshishiku \cite{Tshishiku}.

\subsection{Borel's stable range for the cohomology of $\SL(n,\Z)$}

Here we recall Borel's result.
Let $$N'(\SL(n,\Z),V)=\min(M(\SL(n,\R),V),C'(\SL(n,\Q),V)),$$ where the constant $C'$ is defined later.

\begin{theorem}[Borel \cite{Borel1, Borel2}]\label{Borelbipartition}
(1) For each integer $n\ge 1$, the algebra map $$H^*(\SL(n+1,\Z),\Q)\to H^*(\SL(n,\Z),\Q)$$ induced by the inclusion $\SL(n,\Z)\hookrightarrow \SL(n+1,\Z)$ is an isomorphism in $*\le N'(\SL(n,\Z),\Q)$.
Moreover, we have an algebra isomorphism
\begin{gather*}
    \varprojlim_n H^*(\SL(n,\Z),\Q)\cong \bigwedge_{\Q} (x_1, x_2,\ldots), \quad \deg x_i=4i+1
\end{gather*}
in degree $*\le N'(\SL(n,\Z),\Q)$.

(2)
 Let $V$ be an algebraic $\SL(n,\Q)$-representation such that $V^{\SL(n,\Q)}=0$.
 Then we have 
 \begin{gather*}
     H^p(\SL(n,\Z),V) = 0 \quad \text{for $p\le N'(\SL(n,\Z),V)$}.
 \end{gather*}
\end{theorem}


\subsection{Preliminaries from representation theory}
Before defining Borel's constant, we recall necessary facts from representation theory. See \cite{Fulton-Harris} for details.

Let $n\ge 2$ be an integer. 
Let $\h\subset sl_{n}(\C)$ denote the Cartan subalgebra
$$\h=\{a_1 H_1+\cdots+a_{n}H_{n}\mid a_1+\cdots+a_{n}=0\},$$
where $H_i=E_{i,i}$ is the diagonal matrix.
We write the dual vector space $\h^*$ as
$$\h^*=\C\{L_1,\ldots,L_{n}\}/\C(L_1+\cdots +L_{n}),$$
where $L_i$ is the linear map from the space of diagonal matrices to $\C$ satisfying $L_i(H_j)=\delta_{i,j}$.
The set of \emph{roots} of $sl_{n}(\C)$ is $\{L_i-L_j\mid i\neq j\}$, that of \emph{positive roots} is $\{L_i-L_j\mid i<j\}$ and that of \emph{simple roots} is $\{\alpha_i=L_i-L_{i+1}\mid 1\le i\le n-1\}$.

An element $u=u_1 L_1+\cdots+u_{n} L_{n}$ with $\sum u_i=0$ will be denoted by $[u_1,\ldots, u_{n}]$.
For an element $\phi\in \h^*$, we write $\phi>0$ if
$\phi=\sum_ic_i\alpha_i$ with $c_i>0$ for all $i$.
Note that $\phi=[\phi_1,\ldots,\phi_{n}]\in \h^*$ satisfies $\phi>0$ if and only if $\phi_1+\cdots+\phi_i >0$ for any $i=1,\ldots, n-1$.

The \emph{Weyl group} $W$ of $sl_{n}(\C)$ is the symmetric group $\gpS_{n}=\langle s_1,\ldots, s_{n-1}\rangle$. The generator $s_i$ permutes $L_i$ and $L_{i+1}$ and fixes the other $L_k$.
The \emph{length} $l(\sigma)$ of an element $\sigma\in W$ is the minimum length of
the words in the $s_i$ representing $\sigma$.
Set $W^q=\{\sigma\in W\mid l(\sigma)=q\}$, which consists of elements that send exactly $q$ positive roots to negative roots.
We have $W=\coprod_{q=0}^{l(w_0)} W^q$, where $l(w_0)=\half n(n-1)$ is the length of the longest element $w_0$ of $W=\gpS_n$.

\subsection{Improved constant $C'(\SL(n,\Q),V_{\ul\lambda})$}
Here we define Borel's improved constants $C'$.

For a bipartition $\ul\lambda$ with $l(\ul\lambda)\le n$, let
\begin{gather}
\label{mui}
\mu_{\ul\lambda}=(\mu_1,\ldots,\mu_{n})=(\lambda_1,\ldots, \lambda_{l(\lambda)}, 0,\ldots,0,-\lambda'_{l(\lambda')}, \ldots, -\lambda'_1)
\end{gather}
 be the highest weight of $V_{\ul\lambda}$.
Let $\rho\in \h^*$ be half the sum of positive roots. Then we have $$\rho=[\frac{n-1}{2},\frac{n-3}{2},\frac{n-5}{2},\ldots, -\frac{n-1}{2}]$$
and
$$\rho+\mu_{\ul\lambda}=[\frac{n-1}{2}-\alpha+\mu_1,\frac{n-3}{2}-\alpha+\mu_2,\ldots, -\frac{n-1}{2}-\alpha+\mu_{n}],$$
where $\alpha=\frac{1}{n}\deg\ul\lambda$.
Define 
\begin{gather*}
\begin{split}
C'(\SL(n,\Q),V_{\ul\lambda})
&=\max\{q\in\{0,\dots,l(w_0)\}\mid \sigma(\rho+\mu_{\ul\lambda})>0 \text{ for all }\sigma\in W^{q}\}\ge0.
\end{split}
\end{gather*}
Then we can easily check that 
\begin{equation}\label{C'sym}
    C'(\SL(n,\Q),V_{\ul\lambda})=C'(\SL(n,\Q),V_{\ul\lambda^*}).
\end{equation}

For an algebraic $\SL(n,\Q)$-representation $V$, we set
$$C'(\SL(n,\Q),V)=\min_{\ul\lambda}C'(\SL(n,\Q),V_{\ul\lambda}),$$ where $\ul\lambda$ runs through all bipartitions such that $V_{\ul\lambda}$ is isomorphic to a direct summand of $V$.


\subsection{Estimation of $C'(\SL(n,\Q),V_{\ul\lambda})$}

For each bipartition $\ul\lambda$, we define an integer $\nB(\ul\lambda)\ge 0$ by
$$\nB(\ul\lambda):=\max(2|\deg\ul\lambda|+1, 2l(\lambda), 2l(\lambda')).$$

\begin{theorem}\label{slstablerange}
 Let $n\ge 2$.
  Let $\ul\lambda=(\lambda,\lambda')$ be a bipartition.
  Then for every
  $n\ge l(\ul\lambda)$,
  we have 
  \begin{gather}\label{Borelrangeinequality}
      C'(\SL(n,\Q), V_{\ul\lambda})\le\lfloor n/2\rfloor -1.
  \end{gather} 
  The identity holds if $n\ge \nB(\ul\lambda)$.
\end{theorem}

\begin{proof}
 We first prove \eqref{Borelrangeinequality} for $n\ge l(\ul\lambda)$.
 If $n$ is odd and $(\rho+\mu_{\ul\lambda})_{\frac{n+1}{2}}=-\alpha+\mu_{\frac{n+1}{2}}\ge 0$, then for $\sigma_-=s_{n-1} \cdots s_{\frac{n+1}{2}}\in W$,
 the coefficient of $L_n$ in $\sigma_-(\rho+\mu_{\ul\lambda})$ is 
 $(\rho+\mu_{\ul\lambda})_{\frac{n+1}{2}}\ge 0$.
 If $n$ is odd and $(\rho+\mu_{\ul\lambda})_{\frac{n+1}{2}}< 0$, then for $\sigma_+=s_1 \cdots s_{\frac{n-1}{2}}\in W$,
 the coefficient of $L_1$ in $\sigma_+(\rho+\mu_{\ul\lambda})$ is 
 $(\rho+\mu_{\ul\lambda})_{\frac{n+1}{2}}< 0$.
 If $n$ is even, then we have either $(\rho+\mu_{\ul\lambda})_{\frac{n}{2}}=1/2-\alpha+\mu_{\frac{n}{2}} \ge 0$ or 
 $(\rho+\mu_{\ul\lambda})_{\frac{n}{2}+1}=-1/2-\alpha+\mu_{\frac{n}{2}+1}\le 0$.
 If the former holds, then for $\sigma_-=s_{n-1} \cdots s_{\frac{n}{2}}\in W$,
 the coefficient of $L_n$ in $\sigma_-(\rho+\mu_{\ul\lambda})$ is 
 $(\rho+\mu_{\ul\lambda})_{\frac{n}{2}}\ge 0$.
 If the latter holds, then for $\sigma_+=s_1 \cdots s_{\frac{n}{2}}\in W$,
 the coefficient of $L_1$ in $\sigma_+(\rho+\mu_{\ul\lambda})$ is 
 $(\rho+\mu_{\ul\lambda})_{\frac{n}{2}+1}\le 0$.
 Therefore, in each case, we have $\sigma_{\pm}(\rho+\mu_{\ul\lambda})\not> 0$, which implies \eqref{Borelrangeinequality}.

By \eqref{C'sym}, we have only to consider the case where $\alpha\ge 0$, that is, when $|\lambda|\ge |\lambda'|$.
 Suppose that we have $n\ge \nB(\ul\lambda).$ Thus, we have $0\le \alpha< 1/2$.
 We first prove $C'(\SL(n,\Q), V_{\ul\lambda})=\lfloor n/2\rfloor -1$ for $2\le n\le 4$.
 For $n=2, 3$, this is obvious since we have $\lfloor n/2\rfloor -1=0$.
 For $n=4$, since we have $l(\lambda),l(\lambda')\le2$, it follows that 
 $$\rho+\mu_{\ul\lambda}
 =[3/2+\lambda_1-\alpha, 1/2+\lambda_2-\alpha, -1/2-\lambda'_2-\alpha, -3/2-\lambda'_1-\alpha].$$
 Since $0\le \alpha<1/2$, the first two coefficients are positive and the others are negative.
 For $\sigma=s_1, s_3\in W^1$, it is easily checked that $\sigma(\rho+\mu_{\ul\lambda})>0$. For $\sigma=s_2$, we also have $\sigma(\rho+\mu_{\ul\lambda})>0$ since we have $\lambda_1\ge \lambda'_2$ and thus
 $$(3/2+\lambda_1-\alpha)+(-1/2-\lambda'_2-\alpha)=1-2\alpha+\lambda_1-\lambda'_2\ge 1-2\alpha >0.$$
 Therefore, we have $C'(\SL(n,\Q), V_{\ul\lambda})=\lfloor n/2\rfloor -1$ for $n=4$.
 
 In what follows, we will prove that for $n\ge 5$, $\sigma(\rho+\mu_{\ul\lambda})>0$ for any $\sigma\in W$ of length $\lfloor n/2\rfloor -1$.
 Let $k=1,\ldots,n-1$.
 Considering the inversion number of $\sigma$, we have
 $$
   (\sigma^{-1}(1)-1)+(\sigma^{-1}(2)-2)+\cdots+(\sigma^{-1}(k)-k)\le \lfloor n/2\rfloor-1.
 $$
 Therefore, the sum of the first $k$ coefficients of $\sigma(\rho+\mu_{\ul\lambda})$ is
 \begin{gather*}
  \begin{split}
    &\left(\frac{n-1}{2}-(\sigma^{-1}(1)-1)-\alpha+\mu_{\sigma^{-1}(1)}\right)+\cdots+\left(\frac{n-1}{2}-(\sigma^{-1}(k)-1)-\alpha+\mu_{\sigma^{-1}(k)}\right)\\
    &=\sum_{i=1}^k\left(\frac{n-1}{2}-(i-1)\right)  -\sum_{i=1}^k(\sigma^{-1}(i)-i) +\sum_{i=1}^k(\mu_{\sigma^{-1}(i)}-\alpha)\\
    &\ge\frac{(n-k)k}{2}-(\lfloor n/2\rfloor-1) +\sum_{i=1}^k (\mu_{\sigma^{-1}(i)}-\alpha).
  \end{split}
 \end{gather*}
 Let $T(k)$ denote the right hand side of this inequality.
 It suffices to show that $T(k)>0$ for each $k=1,\ldots, n-1$.
 For $k=n-1$, we have 
 \begin{gather*}
 \begin{split}
T(n-1)&\ge 1/2+\sum_{i=1}^{n-1} \mu_{\sigma^{-1}(i)}-(n-1)\alpha
 =1/2+(|\lambda|-|\lambda'|-\mu_{\sigma^{-1}(n)})-(n-1)\alpha\\
 &=1/2+(n\alpha-\mu_{\sigma^{-1}(n)})-(n-1)\alpha= 1/2+\alpha-\mu_{\sigma^{-1}(n)}\ge 1/2>0.
 \end{split}
 \end{gather*}
 For $1\le k\le n-2$, we have 
  \begin{gather*}
 \begin{split}
  T(k)&=\frac{(n-k)k}{2}-(\lfloor n/2\rfloor-1)-k\alpha +\sum_{i=1}^k \mu_{\sigma^{-1}(i)}\\
  &> \frac{(n-k)k-n+2-k}{2}+\sum_{i=1}^k \mu_{\sigma^{-1}(i)}
  \ge \sum_{i=1}^k \mu_{\sigma^{-1}(i)}.
 \end{split}
 \end{gather*}
 Therefore, it suffices to show that $\sum_{i=1}^k \mu_{\sigma^{-1}(i)}\ge 0.$
 Let $\mathscr{J}=\{j\in \{1,\ldots, \lfloor n/2\rfloor\}\mid \sigma(n+1-j)\le k\}$.
 If $\mathscr{J}=\emptyset$, then  $\sum_{i=1}^k \mu_{\sigma^{-1}(i)}\ge 0$ follows directly from the definition of $\mathscr{J}$. 
Otherwise, let $J=\min \mathscr{J}$.
 Since the length of $\sigma$ is $\lfloor n/2\rfloor-1$, by the hypothesis that $l(\lambda)\le n/2$ and $l(\lambda')\le n/2$, we have
 \begin{gather*}
   \sum_{i=1}^k\mu_{\sigma^{-1}(i)}\ge
   (\mu_1+\cdots+\mu_{\lfloor n/2\rfloor+1-J}) +(\mu_{n+1-J}+\cdots+\mu_{n+1-\lfloor n/2\rfloor}).
 \end{gather*}
 Let $a=\mu_1+\cdots+\mu_{\lfloor n/2\rfloor+1-J}$.
 Then we have $a \ge \frac{\lfloor n/2\rfloor+1-J}{\lfloor n/2\rfloor} |\lambda|$
 since we have 
 $$|\lambda|=a+(\mu_{\lfloor n/2\rfloor+2-J}+\cdots+\mu_{\lfloor n/2\rfloor})
  \le a+ \frac{J-1}{\lfloor n/2\rfloor+1-J} a 
  =\frac{\lfloor n/2\rfloor}{\lfloor n/2\rfloor+1-J}\; a.$$
Let $b=\mu_{n+1-J}+\cdots+\mu_{n+1-\lfloor n/2\rfloor}$. In a similar way, we have $b \ge - \frac{\lfloor n/2\rfloor+1-J}{\lfloor n/2\rfloor} |\lambda'|$.
Therefore, we have
\begin{gather*}
   \sum_{i=1}^k\mu_{\sigma^{-1}(i)}\ge a+b \ge  \frac{\lfloor n/2\rfloor+1-J}{\lfloor n/2\rfloor} (|\lambda|-|\lambda'|)=  \frac{\lfloor n/2\rfloor+1-J}{\lfloor n/2\rfloor} n\alpha\ge 0.
 \end{gather*}
 This completes the proof.
\end{proof}

Note that Theorem \ref{slstablerange} does not give any information for the value of $C'(\SL(n,\Q),V_{\ul\la})$ if $n<\nB(\ul\lambda)$. Some computation suggests the following conjecture, which would completely determine $C'(\SL(n,\Q),V_{\ul\la})$.

\begin{conjecture}
Let $\ul\la$ be a bipartition and let $n\ge l(\ul\la)$ be an integer. 
For $i=1,\dots,n$, set $a(i) = \frac{n+1}2 - i -\alpha+\mu_i$, where $\mu_i$ is given in \eqref{mui}.
Then
we have
\begin{gather*}
    C'(\SL(n,\Q),V_{\ul\la})=\min\{i\in\{1,\dots,n\}\mid 
a(i)\le0\text{ or }a(n+1-i)\ge0\}-2.
\end{gather*}
\end{conjecture}

\subsection{Stable range for the cohomology of $\GL(n,\Z)$}
Now we regard $V_{\ul\lambda}$ as an irreducible algebraic $\GL(n,\Z)$-representation.
We obtain a stable range for cohomology of $\GL(n,\Z)$ with coefficients in $V_{\ul\lambda}$.

For a bipartition $\ul\lambda$ and a non-negative integer $p$, set $$\nB(\ul\lambda,p)=\max(\nB(\ul\lambda), 2p+2)\ge 2.$$
Borel's result (Theorem \ref{Borelbipartition}) and the estimate of Borel's constant $C'(\SL(n,\Q),V_{\ul\la})$ (Theorem \ref{slstablerange}) imply the following.

\begin{corollary}[weaker than Theorem \ref{Borel-Li-Sun}]\label{corollaryhomologyofGL}
  Let $\ul\lambda$ be a bipartition, and let $p\ge0$ and $n\ge \nB(\ul\lambda,p)$ be integers.  Then we have the following.
\begin{enumerate}
\item  If $\ul\lambda=(0,0)$, i.e., $V_{\ul\la}=\Q$, then we have
\begin{gather*}
    H^*(\SL(n,\Z),\Q)\cong H^*(\GL(n,\Z), \Q)\cong \bigwedge_{\Q} (x_1, x_2,\ldots), \quad \deg x_i=4i+1
\end{gather*}
for cohomological degree $*\le p$.
\item  If $\ul\lambda\neq (0,0)$, then we have
  \begin{gather*}
      H^p(\SL(n,\Z), V_{\ul\lambda})= H^p(\GL(n,\Z), V_{\ul\lambda})=0.
  \end{gather*}
\end{enumerate}
\end{corollary}

\begin{proof}
Since we have $n\ge \nB(\ul\lambda,p)\ge \nB(\ul\lambda)\ge 2$, by Theorem \ref{slstablerange}, we have
$C'(\SL(n,\Q), V_{\ul\lambda})=\lfloor n/2\rfloor-1$.
Therefore, we have 
$$C'(\SL(n,\Q), V_{\ul\lambda})=\lfloor n/2\rfloor-1\le n-2\le M(\SL(n,\R),V_{\ul\lambda}).$$
Since we have $n\ge \nB(\ul\lambda,p)\ge 2p+2$, we have 
$$p\le \lfloor n/2\rfloor-1=C'(\SL(n,\Q), V_{\ul\lambda})=N'(\SL(n,\Z), V_{\ul\lambda}).$$
Therefore, the case of $\SL(n,\Z)$ follows from Theorem \ref{Borelbipartition}.

The case of $\GL(n,\Z)$ follows from the case for $\SL(n,\Z)$ and the Hochschild--Serre spectral sequence for the short exact sequence
$$
 1\to \SL(n,\Z)\to \GLnZ\to \Z/{2\Z}\to 1.
$$
\end{proof}

\section{Kupers, Miller and Patzt's method}
\label{secImprovedVanishing}

Kupers, Miller and Patzt \cite{Kupers-Miller-Patzt} improved Borel's original stable range for coefficients in polynomial $\GL(n,\Z)$-representations indexed by partitions.
Here we adapt their arguments to the case of coefficients in algebraic $\GL(n,\Q)$-representations indexed by bipartitions.

\subsection{Polynomial $\VIC$-modules}

There are mutually related theories that can be used in the study of stability of 
sequence of $\GL(n,\Z)$-representations
such as coefficient systems \cite{Dwyer,VanDerKallen,Randal-Williams--Wahl}, representation stability \cite{Church-Farb}, central stability \cite{Putman} 
and $\VIC$-modules \cite{Putman--Sam}.
The notion of polynomiality was introduced by van der Kallen \cite{VanDerKallen} for coefficient systems, and was generalized by Randal-Williams and Wahl \cite{Randal-Williams--Wahl}. See also \cite{Patzt, Kupers-Miller-Patzt}.
Their definition is stronger than Djament and Vespa's strong polynomial functors \cite{Djament--Vespa13}.

Here we recall the notions of $\VIC$-modules and polynomial $\VIC$-modules.

Let $\VIC=\VIC(\Z)$ denote the category of finitely generated free abelian groups and and injective morphisms with chosen complements.
The Hom-set for a pair of objects $M$ and $N$ is given by 
$$\Hom_{\VIC}(M,N)=\{(f,C)\mid f: M\hookrightarrow N,\;
N=\im(f)\oplus C\}.$$
A $\VIC$-module is a functor from $\VIC$ to the category $\Vect_{\Q}$ of $\Q$-vector spaces and linear maps.
A \emph{morphism} (also called a \emph{$\VIC$-module map}) $f:V\to V'$ of $\VIC$-modules $V$ and $V'$ is a natural transformation.  The $\VIC$-modules and morphisms form an abelian category $\VICmod$, as is the case for the category of $\calC$-modules for any essentially small category $\calC$.
The category $\VICmod$ also has a symmetric monoidal category structure whose tensor product is defined objectwise, i.e., $(V\otimes V')(M)=V(M)\otimes V'(M)$ for $M\in\Ob(\VIC)$, and whose monoidal unit is given by the constant functor with value $\Q$.

For each VIC-module $V$ and an integer $n\ge0$, the vector space $V(n)$ is naturally equipped with a $\GLnZ$-module structure.

Let $V$ be a $\VIC$-module.
Define $\VIC$-modules $\ker V$ and $\coker V$ by
\begin{gather*}
\ker V(M):= \ker (V(M)\to V(M\oplus \Z)),\\
\coker V(M):= \coker (V(M)\to V(M\oplus \Z))
\end{gather*}
for any object $M$ of $\VIC$, 
where $V(M)\to V(M\oplus \Z)$ is induced by the inclusion $M\hookrightarrow M\oplus \Z$.
Define the polynomiality of $\VIC$-modules inductively as follows.
Let $m\ge -1$.
We call $V$ \emph{polynomial of degree $-1$ in ranks $> m$} if $V(M)=0$ for any object $M\in \Ob(\VIC)$ with rank $>m$.
For $r\ge 0$, we call $V$ \emph{polynomial of degree $\le r$ in ranks $> m$} if $\ker V$ is polynomial of degree $-1$ in ranks $> \max(m-r-1,-1)$ and if $\coker V$ is polynomial of degree $\le r-1$ in ranks $> \max(m-1,-1)$.
We call $V$ \emph{polynomial of degree (exactly) $r$ in ranks $>m$} if $V$ is polynomial of degree $\le r$ in ranks $>m$ and if $V$ is not polynomial of degree $\le r-1$ in ranks $>m$.
If $m=-1$, then we usually omit ``in ranks $>-1$'' and just write polynomial of degree $\le r$. 

\begin{remark}
Our definition of polynomial VIC-modules is slightly stronger than that in \cite{Kupers-Miller-Patzt} for a technical reason. We need this strengthened definition to have Lemma \ref{lemmapatzt}.
In \cite{Kupers-Miller-Patzt}, they required the condition that if $V$ is polynomial of degree $\le r$ in ranks $> m$, then $\ker V$ is polynomial of degree $-1$ in ranks $> m$ instead of in ranks $>\max(m-r-1,-1)$.
Patzt \cite{Patzt} used a still weaker condition in which $\coker V$ is polynomial of degree $\le r-1$ in ranks $>m$.
If $m=-1$, then all of the three definitions of polynomiality coincide.
\end{remark}

Let $V$ be a polynomial $\VIC$-module of degree $\le r$ in ranks $>m$.
The polynomiality of finite-dimensional $\VIC$-modules imply the polynomiality of dimensions, that is, if $\dim(V(\Z^n))$ is finite for each $n>m$, then 
there is a polynomial $f(x)$ such that we have 
$\dim(V(\Z^n))=f(n)$ for any $n>m$.
Therefore, we obtain the following lemma.

\begin{lemma}\label{deg}
Let $V$ be a polynomial $\VIC$-module of degree $\le r$ in ranks $>m$.
If there exists a polynomial $P(x)$ of degree $r$ such that for each $n>m$, $\dim(V(\Z^n))=P(n)$, then $V$ is polynomial of degree $r$ in ranks $>m$.
\end{lemma}

We have the following degree-change lemma.

\begin{lemma}\label{degreechange}
Let $s\ge r$ and $n\ge m+s-r$.
If $V$ is a polynomial $\VIC$-module of degree $\le r$ in ranks $>m$, then $V$ is polynomial of degree $\le s$ in ranks $>n$.
\end{lemma}

\begin{proof}
If $V$ is polynomial of degree $\le r$ in ranks $>m$, then 
$\ker V$ is polynomial of degree $-1$ in ranks $>\max(m-r-1,-1)$.
Since we have $n-s-1\ge m-r-1$, $\ker V$ is polynomial of degree $-1$ in ranks $>\max(n-s-1,-1)$.
Also $\coker V$ is polynomial of degree $\le r-1$ in ranks $>\max(m-1,-1)$.
Since we have $s\ge r$ and $n\ge m+(s-r)\ge m$, $\coker V$ is polynomial of degree $\le s-1$ in ranks $>\max(n-1,-1)$.
Therefore, $V$ is polynomial of degree $\le s$ in ranks $>n$.
\end{proof}

Then we have the following properties of polynomial $\VIC$-modules, in which we adapt Patzt's lemma to our definition of polynomiality.

\begin{lemma}[Cf. Patzt {\cite[Lemma 7.3]{Patzt}}]\label{lemmapatzt}
 (a) Let $m,r\ge-1$ be integers.
  Let
 $$0\to V'\to V\to V''\to 0$$
 be an exact sequence in ranks $> \max(m-r-1,-1)$, that is, we have an exact sequence
 \begin{gather}\label{shortexact}
     0\to V'(M)\to V(M)\to V''(M)\to 0
 \end{gather}
 for each object $M\in \Ob(\VIC)$ of rank $>\max(m-r-1,-1)$.
 If two of the above three $\VIC$-modules are polynomial of degree $\le r$ in ranks $>m$, then so is the third.
 
 (b)
Let $V$ and $W$ be polynomial $\VIC$-modules of degrees $\le r$ and $\le s$, respectively.
Then the tensor product 
$V\otimes W$ has polynomial degree $\le r+s$.
\end{lemma}

\begin{proof}
(b) is essentially a special case of \cite[Lemma 7.3. (b)]{Patzt}.

Let us prove (a) by adapting the proof of \cite[Lemma 7.3 (a)]{Patzt}.
The case of $r=-1$ is obvious. We use induction on $r$. Suppose that the case of degree $\le r-1$ holds.
From \eqref{shortexact}, by the Snake Lemma, we have an exact sequence
\begin{gather}\label{kercokerex}
    0\to \ker V'\to \ker V\to \ker V'' \to \coker V'\to \coker V\to \coker V''\to 0
\end{gather}
in ranks $>\max(m-r-1,-1)$.

If $V''$ is polynomial of degree $\le r$ in ranks $>m$, then we have $\ker V''=0$ in ranks $>\max(m-r-1,-1)$.
Therefore, the exact sequence \eqref{kercokerex} splits into two short exact sequences in ranks $> \max(m-r-1,-1)$
\begin{gather}\label{kerex}
    0\to \ker V'\to \ker V\to \ker V''\to 0,
\end{gather}
\begin{gather}\label{cokerex}
    0\to \coker V'\to \coker V\to \coker V''\to 0.
\end{gather}

Suppose $V$ and $V'$ are polynomial of degree $\le r$ in ranks $>m$. Then $\coker V'$ is polynomial of degree $\le r-1$ in ranks $>m-1$.
Therefore, we have $\ker\coker V'=0$ in ranks $>\max(m-r-1,-1)$.
By the following commutative diagram
\begin{gather*}
\xymatrix{
\ker V'' \ar[r] & \coker V'\\
\ker^2 V'' \ar[r]\ar[u]^{=} & \ker\coker V'=0 , \ar[u]
}
\end{gather*}
we also have two short exact sequences \eqref{kerex} and \eqref{cokerex}.

Since two of $\ker V'$, $\ker V$ and $\ker V''$ are zero in ranks $> \max(m-r-1,-1)$, the other is also zero in ranks $> \max(m-r-1,-1)$ by \eqref{kerex}.
Since two of $\coker V'$, $\coker V$ and $\coker V''$ are polynomial of degree $\le r-1$ in ranks $>\max(m-1,-1)$, and since we have an exact sequence \eqref{cokerex} in ranks $>\max(m-r-1,-1)$, we can use the induction hypothesis, which completes the proof of (a).
\end{proof}

We also need the following lemma.

\begin{lemma}\label{lempolydirectsum}
Let $V$ and $V'$ be $\VIC$-modules such that $V\oplus V'$ is polynomial of degree $\le r$ in ranks $>m$.
Then both $V$ and $V'$ are polynomial of degree $\le r$ in ranks $>m$.
\end{lemma}

\begin{proof}
We use an induction on degree $r$.
It is easy to check the case of $r=-1$ and any $m\ge -1$.
Suppose that the statement holds for degree $\le r-1$ and any $m\ge -1$.
Let $V\oplus V'$ be a polynomial $\VIC$-module of degree $\le r$ in ranks $>m$.
Then we have $\ker(V\oplus V')=0$ in ranks $>\max(m-r-1,-1)$, and $\coker(V\oplus V')$ is polynomial of degree $\le r-1$ in ranks $>\max(m-1,-1)$.
Since we have $\ker(V\oplus V')=(\ker V)\oplus(\ker V')$, it follows that $\ker V= \ker V'=0$ in ranks $>\max(m-r-1,-1)$.
Since we also have $\coker(V\oplus V')=(\coker V)\oplus(\coker V')$, it follows that $(\coker V)\oplus(\coker V')$ is polynomial of degree $\le r-1$ in ranks $>\max(m-1,-1)$.
Therefore, we can replace $V$ and $V'$ with $\coker V$ and $\coker V'$, respectively, and by using the induction hypothesis, it follows that $V$ and $V'$ are polynomial of degree $\le r$ in ranks $>m$.
\end{proof}

\subsection{Polynomial $\VIC$-module $V^{\langle p,q\rangle}$ of traceless tensors}

For an object $M$ of $\VIC$, which is a free abelian group, let $M_{\Q}=M\otimes_{\Z} \Q$.

Let $V^{1,0}$ (resp. $V^{0,1}$) denote the $\VIC$-module such that $V^{1,0}(M)=M_{\Q}$ (resp. $V^{0,1}(M)=M_{\Q}^{*}$) for an object $M$ of $\VIC$.
Then the $\VIC$-modules $V^{1,0}$ and $V^{0,1}$ are polynomial of degree $1$ by \cite[Definition 3.3]{Randal-Williams}.
For $p,q\ge 0$, let 
$$V^{p,q}=(V^{1,0})^{\otimes p} \otimes (V^{0,1})^{\otimes q},\quad M\mapsto M_{\Q}^{p,q}$$
denote the tensor product of copies of $\VIC$-modules $V^{1,0}$ and $V^{0,1}$.
By Lemma \ref{lemmapatzt}, the $\VIC$-module $V^{p,q}$ is polynomial of degree $\le p+q$, which is actually of degree $p+q$ by Lemma \ref{deg} since we have $\dim(V^{p,q}(n))=n^{p+q}$ for $n\ge 0$.

The construction of the $\GL(n,\Q)$-module $H^{\langle p,q\rangle}$ in Section \ref{alg-glnz-rep} can be extended into $\VIC$-modules as follows.

Note that there is a VIC-module map
\begin{gather*}
    c: V^{1,1}\to V^{0,0},
\end{gather*}
such that for each $M\in\Ob(\VIC)$ the map $c_M:V^{1,1}(M)\to V^{0,0}(M)$ is the evaluation map $M_\Q\ot M_\Q^*\to \Q$, $v\ot f\mapsto f(v)$.
Combining this VIC-module map $c$ with identity and symmetry VIC-module maps in the symmetric monoidal category $\VICmod$, we obtain a $\VIC$-module map
\begin{gather*}
    c_{i,j}: V^{p,q}\to V^{p-1,q-1},
\end{gather*}
for each pair $(i,j)\in [p]\times [q]:=\{1,\dots,p\}\times \{1,\dots,q\}$,
where $(c_{i,j})_M: M_{\Q}^{p,q}\to M_{\Q}^{p-1, q-1}$ is the contraction map $c_{i,j}$ in \eqref{cij} for each $M\in\Ob(\VIC)$.

Let $V^{\langle p,q\rangle}$ denote the \emph{traceless part} of $V^{p,q}$, which is the $\VIC$-module defined by
$$
  V^{\langle p,q\rangle}=\ker\left(\bigoplus_{(i,j)\in [p]\times [q]} c_{i,j} \;:\; V^{p,q}\to \bigoplus_{(i,j)\in [p]\times [q]}V^{p-1,q-1}\right)\subset V^{p,q}.
$$
For each $M\in\Ob(\VIC)$, we have $V^{\langle p,q\rangle}(M)=M_\Q^{\langle p,q\rangle}$.
For $(p,q)=0$, we have $V^{\langle 0,0\rangle}(M)=V^{0,0}(M)=\Q$.

For $0\le l\le \min(p,q)$, let
$$\Lambda_{p,q}(l)=\{((i_1,j_1),\ldots, (i_{l+1},j_{l+1}))\in ([p]\times [q])^{l+1}\mid \substack{1\le i_1<i_2<\cdots<i_{l+1}\le p,\\ j_1,  j_2, \ldots, j_{l+1}: \text{ distinct}}\}.$$
For $I=((i_1,j_1),\ldots, (i_{l+1},j_{l+1}))\in \Lambda_{p,q}(l)$, let 
$$c_{I}: V^{p,q}\to  V^{p-l-1,q-l-1}$$
denote the VIC-module map that is obtained as the composition of contraction maps defined by
\begin{gather*}
    \begin{split}
&(v_1\otimes\cdots\otimes v_{p}) \otimes(f_1\otimes\cdots\otimes f_{q})\\
&\mapsto \left(\prod_{r=1}^{l+1}\langle v_{i_r},f_{j_r}\rangle\right)
(v_1\otimes\cdots\widehat{v_{i_1}}\cdots\widehat{v_{i_{l+1}}}\cdots\otimes v_{p})\otimes (f_1\otimes\cdots\widehat{f_{j_1}}\cdots\widehat{f_{j_{l+1}}}\cdots\otimes f_{q}).
    \end{split}
\end{gather*}

Let $k=\min(p,q)$. Define an increasing filtration $F^*=\{F^l\}_{0\le l\le k}$ 
$$V^{\langle p,q\rangle}=F^0\subset F^1\subset \cdots \subset F^l\subset F^{l+1} \subset \cdots \subset F^k=V^{p,q}$$
of the $\VIC$-module $V^{p,q}$ by
\begin{gather}\label{filtration}
     F^l=\ker \left(\bigoplus_{I\in \Lambda_{p,q}(l)} c_{I}\;:\; V^{p,q}\to \bigoplus_{I\in \Lambda_{p,q}(l)} V^{p-l-1,q-l-1}\right).
\end{gather}
Then we have the following exact sequence in ranks $> p+q-1$
\begin{gather}\label{exactfilt}
    0\to F^{l-1}\to F^{l} \to (V^{\langle p-l,q-l\rangle})^{\oplus \binom{p}{l}\binom{q}{l} l!}\to 0
\end{gather} 
for $1\le l\le k$.
By using \eqref{exactfilt}, we can easily check that for $n\ge p+q$
\begin{gather}\label{dimtraceless}
    \dim(V^{\langle p,q\rangle}(n))= \sum_{i=0}^{\min(p,q)}(-1)^i\binom{p}{i}\binom{q}{i}i!\; n^{p+q-2i},
\end{gather}
which is a monic polynomial in $n$ of degree $p+q$.

\begin{proposition}\label{tracelesspoly}
The $\VIC$-module $V^{\langle p,q\rangle}$ is polynomial of degree $p+q$ in ranks $>2(p+q)$.
\end{proposition}

\begin{proof}
By symmetry, we may assume $p\ge q$.
We prove that $V^{\langle p,q\rangle}$ is polynomial of degree $\le p+q$ in ranks $>2(p+q)$ by induction on $q$.
If $q=0$, then for any $i\ge 0$, the $\VIC$-module $V^{\langle i,0\rangle}=V^{\otimes i}$ is polynomial of degree $\le i$ in ranks $>-1$, hence in ranks $> 2i$.
Suppose that $V^{\langle i,j\rangle}$ is polynomial of degree $\le i+j$ in ranks $>2(i+j)$ for any $j\le q-1$ and $i\ge j$.
Then we have a filtration $F^*=\{F^l\}_{0\le l\le q}$ of $V^{\langle p,q\rangle}$ defined in \eqref{filtration}.
Here we use descending induction on $l$.
For $l=q$, we have $F^l=F^q=V^{p,q}$, which is polynomial of degree $\le p+q$.
Suppose that $F^{l}$ is polynomial of degree $\le p+q$ in ranks $>2(p+q)$.
Then from \eqref{exactfilt}, we have an exact sequence
$$
 0\to F^{l-1}\to F^{l}\to (V^{\langle p-l,q-l\rangle})^{\oplus \binom{p}{l}\binom{q}{l} l!} \to 0
$$
in ranks $> p+q-1= 2(p+q)-(p+q)-1$.
By induction hypothesis, $V^{\langle p-l,q-l\rangle}$ is polynomial of degree $\le p+q-2l$ in ranks $> 2(p+q-2l)$. By Lemma \ref{degreechange}, $V^{\langle p-l,q-l\rangle}$ is also polynomial of degree $\le p+q$ in ranks $> 2(p+q)$.
Hence, by Lemma \ref{lemmapatzt}, $(V^{\langle p-l,q-l\rangle})^{\oplus \binom{p}{l}\binom{q}{l} l!}$ is polynomial of degree $\le p+q$ in ranks $>2(p+q)$.
Since $F^{l}$ is polynomial of degree $\le p+q$ in ranks $> 2(p+q)$, by Lemma \ref{lemmapatzt}, so is $F^{l-1}$.
Therefore, by induction, we see that $V^{\langle p,q\rangle}$ is a polynomial $\VIC$-module of degree $\le p+q$ in ranks $>2(p+q)$.
By Lemma \ref{deg} and \eqref{dimtraceless}, the polynomial degree of $V^{\langle p,q\rangle}$ is $p+q$.
\end{proof}

\begin{remark}
(1) The range ``$>2(p+q)$'' is not optimal.
For example, if $p=0$ or $q=0$, then we have $V^{\langle p,q\rangle}=V^{p,q}$, which is polynomial of degree $p+q$ in ranks $>-1$.

(2) Proposition \ref{tracelesspoly} holds also with the definitions of polynomiality given in Patzt \cite{Patzt} and Kupers--Miller--Patzt \cite{Kupers-Miller-Patzt}, since our definition of polynomiality is stronger than theirs.  The range ``$>2(p+q)$'' could be improved if we use their definitions.
If we use Patzt's definition, we see that the range is $> p+q-1$ by the same argument as the proof of Proposition \ref{tracelesspoly}.
\end{remark}

\subsection{Polynomial $\VIC$-module $V_{\ul\lambda}$}

Here we construct a polynomial $\VIC$-module $V_{\ul\lambda}$ for each bipartition $\ul\lambda$.

For a bipartition $\ul\lambda=(\lambda,\lambda')$, let $p=|\lambda|$ and $q=|\lambda'|$.
Define a $\VIC$-module $V_{\ul\lambda}$ as
$$
 V_{\ul\lambda}(M)
 = V^{\langle p,q\rangle}(M) \otimes_{\Q[\gpS_{p}\times \gpS_{q}]} (S^{\lambda}\otimes S^{\lambda'})
$$
for $M\in\Ob(\VIC)$,
where $\Q[\gpS_p \times \gpS_q]$ acts on $V^{\langle p,q\rangle}(M)$ on the right by permutation of tensor factors.
Note that the $\GLnZ$-module $V_\ulla(\Z^n)$ is isomorphic to the $\GLnZ$-module $V_\ulla(n)$ defined in \eqref{defVlambda}.

The direct sum decomposition \eqref{kercontraction} of $\GL(n,\Z)$-modules can be extended to
the following.

\begin{lemma}
\label{Vpq-decomposition}
We have a direct sum decomposition of $V^{\langle p,q\rangle}$ as a $\VIC\times (\gpS_p\times \gpS_q)$-module
\begin{gather}\label{kercontractionVIC}
    V^{\langle p,q\rangle}=\bigoplus_{\substack{\ul\lambda=(\lambda,\lambda') \\|\lambda|=p,\;|\lambda'|=q}} V_{\ul\lambda}\otimes (S^{\lambda}\otimes S^{\lambda'}).
\end{gather}
\end{lemma}

\begin{proof}
In the category of $\VIC\times(\gpS_p\times\gpS_q)$-modules, we have
\begin{gather*}
    \begin{split}
        V^{\langle p,q\rangle}
        &\cong V^{\langle p,q\rangle}\otimes_{\Q[\gpS_p\times\gpS_q]}(\Q[\gpS_p]\otimes\Q[\gpS_q])\\
        &\cong V^{\langle p,q\rangle}\otimes_{\Q[\gpS_p\times\gpS_q]}\left(\bigoplus_{|\lambda|=p}S^{\lambda}\otimes S^{\lambda}\right)\otimes\left(\bigoplus_{|\lambda'|=q}S^{\lambda'}\otimes S^{\lambda'}\right)\\
        &\cong\bigoplus_{|\lambda|=p,\;|\lambda'|=q}
        \left(V^{\langle p,q\rangle}\otimes_{\Q[\gpS_p\times\gpS_q]}(S^{\lambda}\otimes S^{\lambda'})\right)\otimes(S^{\lambda}\otimes S^{\lambda'})\\
        &\cong\bigoplus_{|\lambda|=p,\;|\lambda'|=q}
        V_{\ul\lambda}\otimes(S^{\lambda}\otimes S^{\lambda'}).
    \end{split}
\end{gather*}
\end{proof}

The following lemma should be well known, but we sketch a proof here since we could not find a suitable reference.
\begin{lemma}
For each bipartition $\ulla$, there is a polynomial $f_\ulla(x)$ of degree $|\ulla|$ such that 
$\dim(V_{\ul\lambda}(n))=f_\ulla(n)$ for $n\ge |\ul\lambda|$.
\end{lemma}

\begin{proof}
Let $n\ge |\ul\lambda|$.
The dimension of $V_{\lambda,0}(n)\otimes V_{0,\lambda'}(n)$ is polynomial of degree $|\ul\lambda|$.
We can obtain the lemma by using induction, since we have a decomposition
$$
V_{\lambda,0}(n)\otimes V_{0,\lambda'}(n)\cong V_{\lambda,\lambda'}(n)\oplus
\bigoplus_{|\mu|<|\la|,\;|\mu'|<|\la'|}
V_{\mu,\mu'}(n)^{\oplus c_{\mu,\mu'}},
$$
where the constants $c_{\mu,\mu'}$, \emph{not} depending on $n$, are determined by the Littlewood--Richardson coefficients (see \cite{Koike}). 
\end{proof}

\begin{proposition}\label{polynomialityofrational}
Let $\ul\lambda=(\lambda,\lambda')$ be a bipartition.
If either $\lambda=0$ or $\lambda'=0$, then the $\VIC$-module $V_{\ul\lambda}$ is polynomial of degree $|\ul\lambda|$ (in ranks $>-1$).
Otherwise, the $\VIC$-module $V_{\ul\lambda}$ is polynomial of degree $|\ul\lambda|$ in ranks $>2|\ul\lambda|$.
\end{proposition}

\begin{proof}
Let $p=|\lambda|$ and $q=|\lambda'|$.

Suppose $p=0$ or $q=0$. Then $V^{\langle p,q\rangle}=V^{p,q}$ is polynomial of degree $\le p+q$.
Since $V_\ulla$ is a direct summand of $V^{p,q}$ by Lemma \ref{Vpq-decomposition}, it follows from Lemma \ref{lempolydirectsum} that $V_\ulla$ is polynomial of degree $\le p+q=|\ulla|$.

Suppose $p,q\neq0$. By Proposition \ref{tracelesspoly}, $V^{\langle p,q\rangle}$ is polynomial of degree $\le p+q$ in ranks $>2(p+q)$.
By Lemma \ref{Vpq-decomposition}, $V_\ulla$ is a direct summand of $V^{\langle p,q\rangle}$. Hence, by Lemma \ref{lempolydirectsum}, it follows that $V_{\ul\lambda}$ is a polynomial $\VIC$-modules of degree $\le p+q$
in ranks $>2(p+q)$.
\end{proof}

\subsection{Irreducibility of $V_\ulla$ in the stable category of VIC-modules}

In this subsection, we make a digression and observe that the VIC-module $V_\ulla$ is an irreducible object in the stable category of VIC-modules, which is independently
known to Powell \cite{Powell-private-communication}.

The functor $V_{\ul\la}$ is \emph{not} an irreducible object in the category $\VICmod$ of $\VIC$-modules since any irreducible object $V$ in $\VICmod$ is concentrated at one rank, i.e., there is an integer $r\ge0$ such that we have $V(n)=0$ for all $n\neq r$.  Thus the VIC-module $V_\ulla$ has infinitely many irreducible subquotients.

We here consider the  \emph{stable category of $\VIC$-modules} defined by Djament and Vespa \cite{Djament--Vespa13}, which is defined as follows.
(See also \cite{Sam-Snowden1, Sam-Snowden2} and Remark \ref{SamSnowden} below.)
A $\VIC$-module $F$ is called \emph{stably zero} if
for any element $x\in F(n)$, $n\ge0$, there is $N\ge n$ such that we have $F(i_{n,N})(x)=0$, where $i_{n,N}:\Z^n\hookrightarrow\Z^N$ is the canonical inclusion.
Let $\Sz$ denote the full subcategory of $\VICmod$ whose objects are stably zero VIC-module.
Then $\Sz$ is a Serre subcategory of $\VICmod$.
The stable category of VIC-modules, $\St $, is 
the quotient abelian category $\St =\VICmod/\Sz$.
Let $\pi:\VICmod\to \St$ denote the canonical functor.

\begin{proposition}[independently known to Powell \cite{Powell-private-communication}]
\label{prop-irreducible}
For each bipartition $\ulla$, the object $\pi(V_{\ul\la})$ is irreducible in $\St $.
\end{proposition}

\begin{proof}
Recall that the $\GLnZ$-module $V_\ulla(\Z^n)$ given by the VIC-module structure coincides with the irreducible $\GLnZ$-module $V_\ulla(n)$ defined in \eqref{defVlambda}.
If $V$ is any VIC-submodule of $V_{\ulla}$ such that $V\neq0$, then there is an integer $N\ge l(\ulla)$ such that we have
\begin{gather*}
V(\Z^n) = \begin{cases}
0&(0\le n< N)\\
V_\ulla(\Z^n) &(N\le n).
\end{cases}
\end{gather*}
Since the quotient VIC-module $V_\ulla/V$ is stably zero, we have an isomorphism $V\cong V_\ulla$ in the stable category $\St $.  Hence $\pi(V_\ulla)$ is irreducible in $\St $.
\end{proof}

\nc\Rep{\mathrm{Rep}}
\begin{remark}
\label{SamSnowden}
Proposition \ref{prop-irreducible} could also be proved by adapting Sam and Snowden's results on $\VIC(\C)$-modules \cite{Sam-Snowden1,Sam-Snowden2}.
In \cite{Sam-Snowden1}, they proved that simple objects of the category $\Rep(\GL)$ of algebraic $\GL_{\infty}(\C)$-modules are classified with bipartitions, and in \cite{Sam-Snowden2} they proved that $\Rep(\GL)$ is equivalent to the stable category of algebraic $\VIC(\C)$-modules.
These results seem to imply that the $\VIC(\C)$-variant of the $\VIC$-module $V_\ulla$ is a simple object in the stable category of $\VIC(\C)$-modules.
\end{remark}

\subsection{Improvement of Borel's vanishing theorem}
In this subsection, we improve Borel's vanishing range for coefficients in the $\VIC$-modules $V_{\ul\lambda}$ by adapting the proof of the following result of Kupers, Miller and Patzt \cite{Kupers-Miller-Patzt}.

\begin{theorem}[Kupers--Miller--Patzt {\cite[Theorem 7.6]{Kupers-Miller-Patzt}}]\label{KMPprop}
Let $\lambda\neq 0$ be a partition.
  Then we have
  $$H_p(\GL(n,\Z),V_{\lambda,0})=0$$
  for $p<n-|\lambda|$.
For the trivial coefficient, we have an isomorphism
\begin{gather*}
    H_p(\GL(n,\Z),\Q)\cong H_p(\GL(n+1,\Z),\Q)
\end{gather*}
for $p<n$.
\end{theorem}

In order to prove Theorem \ref{KMPprop}, Kupers, Miller and Patzt used the homology
$H_p(\GL(n,\Z),\GL(n-1,\Z); V(\Z^n),V(\Z^{n-1}))$
of the pair $(\GL(n-1,\Z), \GL(n,\Z))$ defined as follows.
(See \cite{Dwyer} for details of the construction.)
Let $R_i$ be a projective resolution of $\Z$ over $\Z[\GL(i,\Z)]$ for $i=n-1,n$. 
Then there is a chain map
$$\Phi:R_{n-1}\otimes_{\GL(n-1,\Z)} V(\Z^{n-1})\to R_{n}\otimes_{\GL(n,\Z)}V(\Z^{n})$$
induced by the chain map $R_{n-1}\to R_{n}$ induced by $\id_{\Z}$, which is unique up to chain homotopy.
The homology
$H_p(\GL(n,\Z),\GL(n-1,\Z); V(\Z^n),V(\Z^{n-1}))$
is defined as the homology of the mapping cone $C_*(\Phi)$ of the chain map $\Phi$. 
We have the following.

\begin{theorem}[Kupers--Miller--Patzt {\cite[Theorem 7.3]{Kupers-Miller-Patzt}}]\label{KMPlemma}
Let $V$ be a polynomial $\VIC$-module of degree $r$ in ranks $>m$.
Then we have
$$H_p(\GL(n,\Z),\GL(n-1,\Z); V(\Z^n),V(\Z^{n-1}))=0$$
for $p<n-\max(r,m)$.
Consequently, the inclusion $\GL(n-1,\Z)\to\GL(n,\Z)$ induces an isomorphism
\begin{gather*}
    H_p(\GL(n-1,\Z),V(\Z^{n-1}))\overset{\cong}{\longrightarrow} H_p(\GL(n,\Z),V(\Z^{n}))
\end{gather*}
for $p<n-1-\max(r,m)$.
\end{theorem}

Here we adapt Theorem \ref{KMPlemma} and generalize Theorem \ref{KMPprop} to the $\VIC$-module $V_{\ul\lambda}$.
For a bipartition $\ul\lambda=(\la,\la')$ and a non-negative integer $p$, set 
$$
\nKMP(\ul\lambda,p)=
\begin{cases}
p+1+|\ul\lambda| & (\text{if } \la=0\text{ or }\la'=0)\\
p+1+2|\ul\lambda| & (\text{otherwise}).
\end{cases}
$$

\begin{theorem}[weaker than Theorem \ref{Borel-Li-Sun}]\label{KMPtheorem}
  Let $\ul\lambda\neq (0,0)$ be a bipartition.
  Then we have
  $$H^p(\GL(n,\Z),V_{\ul\lambda})=0$$
  for $n\ge \nKMP(\ul\lambda,p)$.
\end{theorem}

\begin{proof}
By Proposition \ref{polynomialityofrational} and Theorem \ref{KMPlemma}, we have 
$$
 H_{p-1}(\GL(n-1,\Z),V_{\ul\lambda}(\Z^{n-1}))\cong H_{p-1}(\GL(n,\Z),V_{\ul\lambda}(\Z^{n}))
$$
for $p< n- |\ul\lambda|$ if $\la=0$ or $\la'=0$, and for $p< n- 2|\ul\lambda|$ otherwise.
By Corollary \ref{corollaryhomologyofGL}, we have $H^p(\GL(n,\Z),V_{\ul\lambda})=0$ for $n\ge p+1+|\ul\lambda|$ if $\la=0$ or $\la'=0$, and for $n\ge p+1+2|\ul\lambda|$ otherwise.
\end{proof}



\end{document}